\documentclass[11pt,a4paper,twoside]{scrartcl}

\usepackage{amsfonts}
\usepackage{amsmath}
\usepackage{amssymb}
\usepackage{amsthm}

\usepackage{graphicx}
\graphicspath{{pics/}}

\usepackage{enumerate}
\usepackage{verbatim}

\usepackage{
	float,		
	subcaption, 
	xfrac,		
	colortbl,
	mathtools,  
	pifont,  	
	mathrsfs	
}

\usepackage[symbol]{footmisc}
\usepackage{enumitem} 

\usepackage{hyperref}
\usepackage{todonotes}

\usepackage{pgfplots}
\pgfplotsset{compat=1.8}
\usepackage{tikz}
\usepackage{tikz-cd}

\newtheorem{theorem}{Theorem}[section]

\newtheorem{remark}[theorem]{Remark}

\newtheorem{example}[theorem]{Example}
\newtheorem{corollary}[theorem]{Corollary}

\newtheorem{alg}[theorem]{Algorithm}

\makeatletter
\def\imod#1{\allowbreak\mkern10mu({\operator@font mod}\,\,#1)}
\makeatother

\makeatletter 
\let\c@algorithm\c@theorem  
\makeatother

\numberwithin{equation}{section}
\numberwithin{table}{section}
\numberwithin{figure}{section}

\newcommand{\e}{\mathrm e}
\renewcommand{\i}{\mathrm i}
\newcommand{\sinc}{\mathrm{sinc}}
\renewcommand{\b}{\boldsymbol}

\newcommand{\R}{\mathbb R}
\newcommand{\C}{\mathbb C}
\newcommand{\Z}{\mathbb Z}
\newcommand{\N}{\mathbb N}
\newcommand{\T}{\mathbb T}

\newcommand{\I}{\mathcal I}

\newcommand{\ex}{\hspace*{0ex} \hfill \hbox{\vrule height
		1.5ex\vbox{\hrule width 1.4ex \vskip 1.4ex\hrule  width 1.4ex}\vrule
		height 1.5ex}}

\long\def\symbolfootnote[#1]#2{\begingroup%
	\def\thefootnote{\fnsymbol{footnote}}\footnote[#1]{#2}\endgroup}

\allowdisplaybreaks

\title{Optimal density compensation factors for the reconstruction of the Fourier transform of bandlimited functions}

\date{}
\author{Melanie Kircheis\footnotemark[1] \and Daniel Potts\footnotemark[3]}

\hypersetup{pdfauthor={Melanie Kircheis, Daniel Potts},
	pdftitle={Optimal density compensation factors for the reconstruction of the Fourier transform of bandlimited functions},
	pdfsubject={},
	pdfcreator = {pdflatex},
	plainpages=false,
	pdfstartview=FitH,       
	pdfview=FitH,            
	pdfpagemode=UseOutlines, 
	bookmarksnumbered=true, 
	bookmarksopen=false,     
	bookmarksopenlevel=0,   
	colorlinks=true,       
	linkcolor=black,
	citecolor=black,
	urlcolor=black}

\begin{document}
	
\maketitle		

\begin{abstract}
	An inverse nonequispaced fast Fourier transform (iNFFT) is a fast algorithm to compute the Fourier coefficients of a trigonometric polynomial from nonequispaced sampling data.
	However, various applications such as magnetic resonance imaging (MRI) are concerned with the analogous problem for bandlimited functions, i.\,e., the reconstruction of point evaluations of the Fourier transform from given measurements of the bandlimited function.
	In this paper, we review an approach yielding exact reconstruction for trigonometric polynomials up to a certain degree, and extend this technique to the setting of bandlimited functions.
	Here we especially focus on methods computing a diagonal matrix of weights needed for sampling density compensation.
	
	\medskip
	\noindent\emph{Key words}:
	inverse nonequispaced fast Fourier transform, nonuniform fast Fourier transform, direct inversion, density compensation, bandlimited functions, iNFFT, NFFT, NUFFT
	\smallskip
	
	\noindent AMS \emph{Subject Classifications}: \text{
		65Txx, 
		65T50, 
		65F05. 
	}
\end{abstract}

\footnotetext[1]{Corresponding author: melanie.kircheis@math.tu-chemnitz.de, Chemnitz University of
	Technology, Faculty of Mathematics, D--09107 Chemnitz, Germany}
\footnotetext[3]{potts@math.tu-chemnitz.de, Chemnitz University of
	Technology, Faculty of Mathematics, D--09107 Chemnitz, Germany}


\section{Introduction}

The nonequispaced fast Fourier transform (NFFT) is a fast algorithm to evaluate a trigonometric polynomial
\begin{align}
\label{eq:trig_poly_2d}
f(\b x) = \sum_{\b k \in \I_{\b M}} \hat{f}_{\b k}\, \e^{2\pi\i \b k \b x}
\end{align}
with given Fourier coefficients \mbox{$\hat f_{\b k}\in\C$}, \mbox{$\b k\in\I_{\b{M}}$},
at non\-equi\-spaced points \mbox{$\b x_j$}, \mbox{$j=1,\dots,N$}, with \mbox{$N\in\N$}, 
where for \mbox{$\b M \coloneqq (M,\dots,M)^T$}, \mbox{$M\in 2\N$}, we define the multi-index set
\mbox{$\I_{\b M} \coloneqq \Z^d \cap \left[-\tfrac{M}{2},\tfrac{M}{2}\right)^d$} 
with cardinality \mbox{$|\I_{\b M}| = M^d$}.
For more information see \cite[pp.~377-381]{PlPoStTa18} and references therein.
In case of equispaced points $\b x_j$ and $|\I_{\b{M}}|=N$, this evaluation can be realized by means of the well-known fast Fourier transform (FFT); an algorithm that is invertible.
However, various applications need to perform an inverse nonequispaced fast
Fourier transform (iNFFT), i.\,e., compute the Fourier coefficients~$\hat f_{\b k}$ from given function evaluations~$f(\b x_j)$ of the trigonometric polynomial~\eqref{eq:trig_poly_2d}.
Hence, we are interested in an inversion also for nonequispaced data.
In contrast to iterated procedures, where multiple applications of the NFFT are needed to compute a solution, we consider so-called direct methods,
where for a fixed set of points $\b x_j$, \mbox{$j=1,\dots, N$}, the reconstruction can be realized with the same number of arithmetic operations as a single application of an adjoint NFFT.
To achieve this, a certain precomputational step is compulsory, since the adjoint NFFT does not yield an inversion of the NFFT by itself, see  \cite[Sec.~3]{KiPo23}. 
Although this precomputations might be rather costly, they need to be done only once for a given set of points 
$\b x_j$, \mbox{$j=1,\dots,N$},
while the actual reconstruction step is very efficient.
Thus, direct methods are especially beneficial in case of multiple applications with fixed points.

For the setting of trigonometric polynomials several approaches are known,
while in this paper we focus on the analogous problem for bandlimited functions with
\begin{align}
\label{eq:forward_integral}
f(\b x) 
= 
\int\limits_{\left[-\frac M2,\frac M2\right)^d} \hat f(\b v)\,\e^{2\pi\i \b v\b x}\,\mathrm d\b v,
\quad \b x\in \R^d .
\end{align}
In various applications such as magnetic resonance imaging (MRI), cf.~\cite{EgKiPo22}, the aim is the reconstruction of point evaluations \mbox{$\hat f(\b k)\in\C$}, \mbox{$\b k\in\I_{\b{M}}$}, of an object $\hat f$ from given measurements \mbox{$f(\b x_j)$}, \mbox{$j=1,\dots,N$}, of the form \eqref{eq:forward_integral}.
To this end, known approaches for the direct inversion of the NFFT shall be extended to this setting.
In particular, we investigate methods computing a diagonal matrix of weights \mbox{$w_j\in\C$}, \mbox{$j=1,\dots,N$}, 
which are needed for sampling density compensation due to the nonequispaced sampling~\mbox{$\b x_j$}, \mbox{$j=1,\dots,N$}.

Therefore, this paper is organized as follows. 
Firstly, in Section~\ref{sec:trig_poly} we review the approach in \cite[Sec.~3.1]{KiPo23}, that leads to an exact reconstruction for all trigonometric polynomials \eqref{eq:trig_poly_2d} of degree~$\b M$.
Subsequently, in Section~\ref{sec:bandlim} we aim for an analogous approach for bandlimited functions \mbox{$f\in L_1(\R^d)\cap C_0(\R^d)$} as well.
Finally, in Section~\ref{sec:numerics} we show some numerical examples to investigate the accuracy of our approaches.

\section{Trigonometric polynomials \label{sec:trig_poly}}

Let \mbox{$\T^d \coloneqq \R^d\setminus\Z^d$} denote the $d$-dimensional torus with \mbox{$d\in\N$}.
The inner product of two vectors shall be defined as usual as
\mbox{$\b k \b x \coloneqq k_1 x_1 + \dots + k_d x_d$}.
For the reconstruction of the Fourier coefficients \mbox{$\hat f_{\b k}\in\C$}, \mbox{$\b k\in\I_{\b{M}}$},
from function evaluations \mbox{$f(\b x_j)$}, \mbox{$j=1,\dots,N$},
of trigonometric polynomials \mbox{$f \in L_2(\T^d)$} of degree $\b M$, see \eqref{eq:trig_poly_2d},
the following result is known, cf.~\cite[Cor.~3.4]{KiPo23}.
\begin{theorem}
	\label{Thm:exact_trig_poly}
	Let \mbox{$|\I_{\b{2M}}| \leq N$} and \mbox{$\b x_j\in\T^d$}, \mbox{$j=1,\dots,N$}, be given.
	Then the density compensation factors \mbox{$w_j\in\C$} satisfying
	\begin{align}
	\label{eq:claim_quadrature_coefficients_double}
	\sum_{j=1}^N w_j \,\e^{2\pi\i\b k\b x_j} = \delta_{\b 0,\b k},
	\quad\b k\in{\I_{\b {2M}}},
	\end{align}
	with the Kronecker delta \mbox{$\delta_{\b \ell,\b k}$}, are optimal,
	since for all trigonometric polynomials of degree~$\b M$, see \eqref{eq:trig_poly_2d},
	an exact reconstruction of the Fourier coefficients \mbox{$\hat f_{\b k}$} is given by
	\begin{align}
	\label{eq:Fourier_coeffs_nonequi}
	\hat f_{\b k}
	=
	h_{\b k}^{\mathrm{w}}
	\coloneqq 
	\sum_{j=1}^{N} 
	w_j\, f(\b x_j)\,\e^{-2\pi\i \b k \b x_j},
	\quad \b k\in\I_{\b M}.
	\end{align}
\end{theorem}

\begin{proof}
	It is known that \mbox{$\{\e^{2\pi\i \b\ell \b x} \colon \b\ell\in\Z^d\}$} forms an orthonormal basis
	of the Hilbert space~$L_2(\T^d)$ of all $1$-periodic, complex-valued functions, see~\cite[p.~161]{PlPoStTa18}.
	To achieve an exact reconstruction for all trigonometric polynomials~\eqref{eq:trig_poly_2d}
	with maximum degree~$\b M$ it suffices to consider the set of basis functions with \mbox{$\b\ell\in\I_{\b M}$}.
	For each of these basis functions with fixed \mbox{$\b\ell\in\I_{\b M}$} we have
	\begin{align*}
	h_{\b k}^{\mathrm{w}}
	= 
	\sum_{j=1}^{N} 
	w_j\,\e^{2\pi\i (\b\ell-\b k) \b x_j},
	\quad \b k\in\I_{\b M},
	\end{align*}
	and
	\begin{align*}
	\hat f_{\b k}
	&= 
	c_{\b k}(f)
	=
	\int\limits_{\T^d} \e^{2\pi\i (\b\ell-\b k) \b x}\,\mathrm d{\b x}
	=
	\delta_{\b\ell,\b k},
	\quad \b k\in\I_{\b M},
	\end{align*}
	where \mbox{$\delta_{\b \ell,\b k}$} denotes the Kronecker delta.
	Thus, to obtain \eqref{eq:Fourier_coeffs_nonequi} we need to assure that the weights satisfy
	\begin{align}
	\label{eq:claim_FK}
	\sum_{j=1}^{N} 
	w_j\,\e^{2\pi\i (\b\ell-\b k) \b x_j}
	=
	\delta_{\b\ell,\b k},
	\quad \b\ell,\b k\in\I_{\b M}.
	\end{align}
	Since for \mbox{$\b k,\b\ell\in\I_{\b M}$} we have \mbox{$(\b\ell-\b k) \in\I_{\b{2M}}$}, the property \eqref{eq:claim_FK} is fulfilled by \eqref{eq:claim_quadrature_coefficients_double}.
	In other words, by defining the nonequispaced Fourier matrix
	\begin{align}
	\label{eq:matrix_A}
	\b A = \b A_{|\I_{\b{M}}|} \coloneqq \left( \e^{2\pi\i \b k \b x_j} \right)_{j=1,\,\b k \in \I_{\b M}}^{N} 
	\ \in \mathbb C ^{N\times |\I_{\b M}|}
	\end{align}
	as well as \mbox{$\b w \coloneqq (w_j)_{j=1}^N$}, an exact solution to the linear system of equations
	\begin{align}
	\label{eq:exactness_cond_trig_poly}
	\b A_{|\I_{\b {2M}}|}^T \,\b w = \b e_{\b 0} \coloneqq \left( \delta_{\b 0,\b k} \right)_{\b k\in{\I_{\b{2M}}}}
	\end{align}
	yields an exact reconstruction for all trigonometric polynomials \eqref{eq:trig_poly_2d} of degree $\b M$.
	As already mentioned in \cite[Sec. 3.1]{Hu09} an exact solution to \eqref{eq:exactness_cond_trig_poly} can only be found
	if \eqref{eq:exactness_cond_trig_poly} is an underdetermined system with \mbox{$|\I_{\b{2M}}| \leq N$}. 
\end{proof}

Thus, in case \mbox{$|\I_{\b{2M}}| \leq N$} and the condition number of \mbox{$\b A_{|\I_{\b {2M}}|}$} is sufficiently small, the optimal density compensation factors \mbox{$\b w\in\C^N$} can be found as solution to
the linear system of equations \eqref{eq:exactness_cond_trig_poly}.
As mentioned in \cite[Sec.~3.1.1]{KiPo23} the unique solution is then given by the normal equations of second kind 
\begin{align}
\label{eq:normal_equations_second_kind_double}
\b A_{|\I_{\b {2M}}|}^T \overline{\b A_{|\I_{\b {2M}}|}} \,\b v = \b e_{\b 0},
\quad 
\overline{\b A_{|\I_{\b {2M}}|}} \,\b v = \b w,
\end{align}
and can efficiently be computed by an iteration procedure combining the CG algorithm and the NFFT.

However, when \mbox{$|\I_{\b{2M}}| > N$} there is no theoretical guaranty.
In addition, the numerics in \cite{KiPo23} have shown that the least squares solution satisfying the normal equations of first kind
\begin{align}
\label{eq:normal_equations_first_kind_double}
\overline{\b A_{|\I_{\b {2M}}|}} \,\b A_{|\I_{\b {2M}}|}^T \,\b w = \overline{\b A_{|\I_{\b {2M}}|}} \,\b e_{\b 0} ,
\end{align}
is not a good approximation.
This is why, we recommend another computation scheme in this setting.
To this end, note that \eqref{eq:exactness_cond_trig_poly} implies by \eqref{eq:claim_FK} that
\begin{align}
\label{eq:lgs_wcf}
\b A^* \b W \b A = \b I_{|\I_{\b M}|} 
\end{align}
with the weight matrix \mbox{$\b W \coloneqq \mathrm{diag} (w_j)_{j=1}^N$}.
Hence, in case a solution to \eqref{eq:exactness_cond_trig_poly} cannot be found using \eqref{eq:normal_equations_second_kind_double}, 
one may try to find suitable weights utilizing \eqref{eq:lgs_wcf} by minimizing the Frobenius norm 
\mbox{$\big\| \b A^* \b W \b A - \b I_{|\I_{\b M}|} \big\|_{\mathrm F}^2$}.
As shown in \cite[Sec.~3.4.2]{KiPo23} this minimizer can be obtained by solving \mbox{$\b S \b w = \b b$}, 
where
\begin{align}
\label{eq:wcf_system_matrix}
\b S 
&\coloneqq
\Big( \left| \left[ \b A \b A^* \right]_{j,s} \right|^2 \Big)_{j,s=1}^N
\quad\text{ and }\quad
\b b 
= 
|\I_{\b M}| \cdot \b 1_N .
\end{align}
Note that by means of the definition \eqref{eq:matrix_A} as well as 
\mbox{$\b f \coloneqq \left(f(\b x_j)\right)_{j=1}^N$} and \mbox{$\b{\hat f} \coloneqq (\hat f_{\b k})_{\b k\in\I_{\b{M}}}$}, 
the reconstruction \eqref{eq:Fourier_coeffs_nonequi} can be denoted as \mbox{$\b{\hat f} = \b A^* \b W \b f$}.
For further details see \cite{KiPo23}.

\section{Bandlimited functions \label{sec:bandlim}}

Since various applications 
are concerned with bandlimited functions instead of trigonometric polynomials,
we now aim to extend the method from Section~\ref{sec:trig_poly} to bandlimited functions \mbox{$f\in L_1(\R^d)\cap C_0(\R^d)$} with maximum bandwidth $\b M$, i.\,e., functions whose (continuous) Fourier transform
\begin{align}
\label{eq:inverse_integral}
\hat f(\b v)
\coloneqq 
\int\limits_{\R^d} 
f(\b x)\,\e^{-2\pi\i \b v\b x}\,\mathrm d\b x,
\quad \b v\in\R^d,
\end{align}
is supported on \mbox{$\left[-\frac M2,\frac M2\right)^d$}.
Thus, their inverse Fourier transform 
\begin{align*}
f(\b x) 
= 
\int\limits_{\R^d} \hat f(\b v)\,\e^{2\pi\i \b v\b x}\,\mathrm d\b v,
\quad \b x\in \R^d,
\end{align*}
can be written as in~\eqref{eq:forward_integral}.

\subsection{Reconsideration as trigonometric polynomials}

To find a suitable reconstruction technique for a bandlimited function $f$, we consider its 1-periodized version
\begin{align*}
\tilde f(\b x) 
\coloneqq
\sum_{\b r\in\Z^d} f(\b x+\b r) .
\end{align*}
Note that \mbox{$\tilde f\in L_2(\T^d)$} is uniquely representable in form of its absolute convergent Fourier series
\begin{align}
\label{eq:periodization_f}
\tilde f(\b x) 
\coloneqq
\sum_{\b k\in\Z^d} c_{\b k}(\tilde f) \,\e^{2\pi\i\b k\b x} ,
\end{align}
see~\cite[Thm.~4.5]{PlPoStTa18}, where the Fourier coefficients are given by
\begin{align}
\label{eq:fourier_coeffs_ftilde}
c_{\b k}(\tilde f) 
&= 
\int\limits_{\T^d} \tilde f(\b x)\,\e^{-2\pi\i \b k\b x}\,\mathrm d\b x 
= 
\int\limits_{\R^d} f(\b x)\,\e^{-2\pi\i \b k\b x}\,\mathrm d\b x
=
\hat f(\b k),
\quad \b k\in\Z^d,
\end{align}
cf.~\eqref{eq:inverse_integral}.
Moreover, it is known that $f$ is a bandlimited function with bandwidth $\b M$, i.\,e., we have \mbox{$\hat f(\b k) = 0$}, \mbox{$\b k\in\Z^d\setminus\I_{\b{M}}$}.
Therefore, the periodic function \eqref{eq:periodization_f} in fact is a trigonometric polynomial of degree $\b M$ as in \eqref{eq:trig_poly_2d}, which makes it reasonable to utilize the result from Section~\ref{sec:trig_poly} to reconstruct the Fourier coefficients \eqref{eq:fourier_coeffs_ftilde}.
Namely, an exact solution to the linear system \eqref{eq:exactness_cond_trig_poly} yields an exact reconstruction 
\begin{align}
\label{eq:reconstr_exact_bandlim}
\hat f(\b k)
=
c_{\b k}(\tilde f)
=
\sum_{j=1}^{N} 
w_j\, \tilde f(\b x_j)\,\e^{-2\pi\i \b k \b x_j},
\quad \b k\in\I_{\b{M}}.
\end{align}
Note that using the nonequispaced Fourier matrix \mbox{$\b A\in\C^{N\times |\I_{\b{M}}|}$} in \eqref{eq:matrix_A}, the weight matrix \mbox{$\b W = \mathrm{diag} (w_j)_{j=1}^N$} as well as the vectors
\mbox{$\b{\tilde f} \coloneqq (\tilde f(\b x_j))_{j=1}^N$} and \mbox{$\b{\hat f} \coloneqq (\hat f(\b k))_{\b k\in\I_{\b{M}}}$}, 
the reconstruction \eqref{eq:reconstr_exact_bandlim} can be denoted as \mbox{$\b{\hat f} = \b A^* \b W \b{\tilde f}$}.
In addition, by \eqref{eq:periodization_f} we have \mbox{$\b{\tilde f} = \b A \b{\hat f}$}, such that \eqref{eq:lgs_wcf} is fulfilled.

However, in practical applications, such as MRI, this is only a hypothetical case, since the periodization $\tilde f$ cannot be sampled.
Due to a limited coverage of space by the acquisition, the function $f$ is typically only on a bounded domain, w.l.o.g. for \mbox{$\b x\in \left[ -\tfrac 12, \tfrac 12 \right)^d$}, cf.~\cite{EgKiPo22}.
Thus, we need to assume that $f$ is small outside the interval~\mbox{$\left[ -\tfrac 12, \tfrac 12 \right)^d$}, such that \mbox{$\tilde f(\b x_j) \approx f(\b x_j)$},
and we have to deal with the approximation
\begin{align}
\label{eq:approx_fk_nonequi}
\hat f(\b k)
\approx 
\sum_{j=1}^{N} 
w_j\, f(\b x_j)\,\e^{-2\pi\i \b k \b x_j},
\quad \b k\in\I_{\b{M}}.
\end{align}
This is to say, by \eqref{eq:reconstr_exact_bandlim} the error in the approximation \eqref{eq:approx_fk_nonequi} solely occurs because $f$ is not known on whole $\R^d$.
Note that the reconstruction \eqref{eq:approx_fk_nonequi} can also be denoted as \mbox{$\b{\hat f} \approx \b A^* \b W \b f$} with \mbox{$\b W = \mathrm{diag} (w_j)_{j=1}^N$} from \eqref{eq:exactness_cond_trig_poly}.

\begin{remark}
	As known by \cite[Ex.~1.22]{LuBo}, the periodization of the function 
	\mbox{$g(\b t) = \e^{2\pi\i \b t \b x}$}, \mbox{$\b t\in \left[-\frac M2,\frac M2\right)^d$}, 
	with \mbox{$\b x\in\C^d$} fixed, possesses the absolutely and uniformly convergent Fourier series
	\begin{align}
	\label{eq:fourier_series_exp}
	\e^{2\pi\i\b t \b x}
	=
	\sum_{\b\ell\in\Z^d} \e^{2\pi\i\b t \b y_{\b\ell}} \,\sinc\big(M\pi(\b x - \b y_{\b\ell})\big),
	\quad \b x\in\C^d ,
	\end{align}
	with the $d$-variate \mbox{$\sinc$} function \mbox{$\sinc(\b x) \coloneqq \prod_{t=1}^d \mathrm{sinc}(x_t)$}
	with
	\begin{align*}
	\mathrm{sinc}(x) \coloneqq \left\{ \begin{array}{ll}  \frac{\sin x}{x} & \quad x \in \mathbb R \setminus \{0\}\,, \\ [1ex]
	1 & \quad x = 0\,, \end{array} \right.
	\end{align*}
	and equispaced points
	\mbox{$\b y_{\b\ell} \coloneqq \b M^{-1} \odot {\b \ell} = \left( M^{-1} \ell_1, \dots, M^{-1} \ell_d \right)^T$}, 
	\mbox{$\b\ell\in \Z^d$}.
	Since the Fourier coefficients \mbox{$\sinc\big(M\pi(\b x - \b y_{\b\ell})\big)$} in \eqref{eq:fourier_series_exp} are $\ell_2$-summable, we may introduce the $\sinc$ operator
	\begin{align}
	\label{eq:operator_sinc}
	\mathcal C 
	\coloneqq 
	\bigg( \sinc\big(M\pi(\b x_j - \b y_{\b\ell})\big) \bigg)_{j=1,\, \b\ell \in \Z^d}^{N} .
	\end{align}
	By additionally defining the one-sided infinite Fourier matrix 
	\begin{align*}
	\mathcal F 
	\coloneqq 
	\left( \e^{2\pi\i \b k \b y_{\b\ell}} \right)_{\b\ell \in \Z^d,\, \b k \in \I_{\b M}} ,
	\end{align*}
	the matrix product $\mathcal C \mathcal F$ can be written as
	\begin{align}
	\label{eq:operator_CF}
	\mathcal C \mathcal F
	=
	\Bigg(
	\sum_{\b\ell\in\Z^d} \e^{2\pi\i \b k \b y_{\b\ell}} \,\sinc\big(M\pi(\b x_j - \b y_{\b\ell})\big)
	\Bigg)_{j=1,\, \b k \in \I_{\b M}}^{N} 
	=
	\b A
	\end{align}
	with the nonequispaced Fourier matrix \mbox{$\b A\in\C^{N\times |\I_{\b{M}}|}$} in~\eqref{eq:matrix_A},
	since the components of~\eqref{eq:operator_CF} coincide with point evaluations of \eqref{eq:fourier_series_exp}
	at \mbox{$\b x = \b x_j$}, \mbox{$j=1,\dots,N$}, and \mbox{$\b t = \b k\in\I_{\b{M}}$}.
	\ex
\end{remark}

\subsection{Connection to previous work}

Analogous to \cite[Thm.~3.8]{KiPo23} we aim to extend the approximation \eqref{eq:reconstr_exact_bandlim} onto the whole interval, i.\,e., we consider
\begin{align*}
\hat f(\b v)
\approx
\tilde h(\b v) 
\coloneqq 
\sum_{j=1}^{N} 
w_j\, \tilde f(\b x_j)\,\e^{-2\pi\i \b v \b x_j},
\quad \b v\in\left[-\tfrac M2,\tfrac M2\right)^d.
\end{align*}
Inserting this into the inverse Fourier transform~\eqref{eq:forward_integral} yields
\begin{align*}
f(\b x) 
&= 
\int\limits_{\left[-\frac M2,\frac M2\right)^d} \hat f(\b v)\,\e^{2\pi\i \b v\b x}\,\mathrm d\b v
\approx
\int\limits_{\left[-\frac M2,\frac M2\right)^d} 
\tilde h(\b v) \,\e^{2\pi\i \b v\b x}\,\mathrm d\b v \notag \\
&=
\sum_{j=1}^{N} w_j\, \tilde f(\b x_j)
\int\limits_{\left[-\frac M2,\frac M2\right)^d} \e^{-2\pi\i \b v (\b x_j-\b x)} \,\mathrm d\b v \notag \\
&=
\sum_{j=1}^{N} w_j\, \tilde f(\b x_j) \cdot |\I_{\b M}| \, \sinc\big(M\pi(\b x_j-\b x)\big),
\quad \b x \in \R^d.
\end{align*}
Especially, by evaluation at \mbox{$\b x = \b x_s$}, \mbox{$s=1,\dots,N$}, we obtain
\begin{align*}
f(\b x_s) 
&\approx
\sum_{j=1}^{N} w_j\, \tilde f(\b x_j) \cdot |\I_{\b M}| \, \sinc\big(M\pi(\b x_j-\b x_s)\big).
\end{align*}
Therefore, one could also aim to choose suitable weights \mbox{$w_j$} based on this equation.
More precisely, by defining the nonequispaced $\sinc$ matrix 
\begin{align}
\label{eq:matrix_C_nonequi}
\b C_{\mathrm{n}}
\coloneqq 
\bigg( \sinc\big(M\pi(\b x_j - \b x_s)\big) \bigg)_{j,s=1}^{N} 
\ \in \mathbb R^{N\times N} ,
\end{align}
as well as 
\mbox{$\b f = \left(f(\b x_j)\right)_{j=1}^N$} and \mbox{$\b{\tilde f} = (\tilde f(\b x_j))_{j=1}^N$},
one may try to find a weight matrix \mbox{$\b W = \mathrm{diag} (w_j)_{j=1}^N$} such that the approximation
\mbox{$\b f \approx |\I_{\b M}| \cdot \b C_{\mathrm{n}} \b W \b{\tilde f}$}
is best as possible.
Note that still \mbox{$\b{\tilde f}$} is unknown, such that we have to deal with the overall approximation 
\mbox{$\b f \approx |\I_{\b M}| \cdot \b C_{\mathrm{n}} \b W \b{f}$},
i.\,e., one would ideally aim for 
\mbox{$\b I_N = |\I_{\b M}| \cdot \b C_{\mathrm{n}} \b W$}.
We remark that on the main diagonal \mbox{$|\I_{\b M}| \cdot \b C_{\mathrm{n}} \b W = \b I_N$} reads as
\begin{align*}
\tfrac{1}{|\I_{\b M}|} 
= w_j \,\sinc\big(M\pi(\b x_j-\b x_j)\big) %
= w_j \,\sinc(0) = w_j, \quad j=1,\dots,N.
\end{align*}
However, for all other entries with \mbox{$j\neq s$} we would need 
\mbox{$\sinc(M\pi(\b x_j-\b x_s)) = 0$}, which is only true for $\b x_j$ equispaced.
In other words, for arbitrary points $\b x_j$ equality in
\mbox{$|\I_{\b M}| \cdot \b C_{\mathrm{n}} \b W = \b I_N$} 
cannot be fulfilled for any weights.

Hence, one can only look for an approximate solution, e.\,g. by considering the least squares problem
\begin{align*}
\underset{\b W = \mathrm{diag}(w_j)_{j=1}^N}{\text{Minimize }} \ 
\||\I_{\b M}| \cdot \b C_{\mathrm{n}} \b W -\b I_N\|_{\mathrm F}^2.
\end{align*}
When denoting the $j$-th column of $\b W$, $\b I_N$ and $\b C_{\mathrm n}$ 
as $\b w_j$, $\b e_j$ and $[\b C_{\mathrm n}]_j$, respectively,  
we may rewrite the Frobenius norm by only considering the nonzero entries via
\begin{align*}
\||\I_{\b M}| \cdot \b C_{\mathrm n} \b W-\b I_N\|_{\mathrm F}^2
&=
\sum_{j=1}^N \||\I_{\b M}| \cdot \b C_{\mathrm n} \b w_j-\b e_j\|_2^2 
=
\sum_{j=1}^N \left\||\I_{\b M}| \cdot [\b C_{\mathrm n}]_j \, {w_j}-\b e_j\right\|_2^2.
\end{align*}
Thus, as stated in \cite{GrLeIn06} (without proof) the least squares solution to the minimization problem is given by
\begin{align}
\label{eq:sol_greengard}
{w_j}
&=
\frac{1}{|\I_{\b M}|} \, [\b C_{\mathrm n}]_j^\dagger \, \b e_j
=
\frac{1}{|\I_{\b M}|} \left( [\b C_{\mathrm n}]_j^* \, [\b C_{\mathrm n}]_j \right)^{-1} [\b C_{\mathrm n}]_j^* \, \b e_j \notag \\
&=
\frac{\sinc\big(M\pi (\b x_j-\b x_j)\big)}
{|\I_{\b M}|\, \sum_{s=1}^N \,\sinc^2(M\pi(\b x_j-\b x_s))} 
=
\frac{1}{|\I_{\b M}|}
\left( \sum_{s=1}^N \sinc^2(M\pi(\b x_j-\b x_s)) \right)^{-1}. %
\end{align}

\begin{remark}
	Note that in \cite{GrLeIn06} it was claimed that this approach coincides with the one in \cite{PiMe99} only considering finite sections of \eqref{eq:operator_sinc} with \mbox{$\b\ell\in\I_{\b{M}}$}.
	However, we remark that this claim only holds asymptotically for \mbox{$|\I_{\b M}|\to\infty$}. 
	
	By the classical sampling theorem of Shannon-Whittaker-Kotelnikov, see \cite{Whittaker, Shannon49, Kotelnikov},
	any bandlimited function \mbox{$f\in L_2(\R^d)$} with maximum bandwidth~$\b M$ can be recovered from its uniform samples \mbox{$f(\b y_{\b\ell})$}, \mbox{$\b\ell\in\Z^d$}, and we have
	\begin{align}
	\label{eq:shannon}
	f(\b x) 
	= 
	\sum_{\b\ell\in\Z^d} f(\b y_{\b\ell}) 
	\,\sinc\big(M\pi(\b x - \b y_{\b\ell})\big) ,
	\quad \b x\in\R^d .
	\end{align}
	Now we apply the sampling theorem of Shannon-Whittaker-Kotelnikov to the shifted $\sinc$ function
	\mbox{$f(\b x)=\sinc(M\pi(\b x_j-\b x)) \in L_2(\R^d)\cap C(\R^d)$} with $j$ fixed.
	By evaluation at \mbox{$\b x = \b x_s$}, \mbox{$s=1,\dots,N$}, we obtain
	\begin{align}
	\label{eq:shannon_sinc}
	\sinc\big(&M\pi(\b x_j-\b x_s)\big) 
	= 
	\sum_{\b\ell\in\Z^d} \sinc\big(M\pi(\b x_j-\b y_{\b\ell})\big) 
	\,\sinc\big(M\pi(\b x_s - \b y_{\b\ell})\big) .
	\end{align}
	Using the $\sinc$ operator $\mathcal C$ in \eqref{eq:operator_sinc} and the nonequispaced $\sinc$ matrix in \eqref{eq:matrix_C_nonequi}, this can be written as
	\begin{align}
	\label{eq:matrix_Cn_split}
	\b C_{\mathrm{n}} = \mathcal C \mathcal C^* \in \R^{N\times N} .
	\end{align}
	Hence, a restriction to finitely many $\b \ell$ in \eqref{eq:shannon_sinc} corresponds to uniform truncation of a Shannon series \eqref{eq:shannon}, which is known as a poor approximation due to the slow convergence of the $\sinc$ function.
	Thus, equality in \eqref{eq:matrix_Cn_split} is only satisfied for the operator $\mathcal C$ in \eqref{eq:operator_sinc}, while considering finite sections of \eqref{eq:operator_sinc} with \mbox{$\b\ell\in\I_{\b{M}}$} implies a poor approximation of \mbox{$\b C_{\mathrm{n}}$}.
	\ex
\end{remark}

\section{Numerics \label{sec:numerics}}

\begin{figure}[!h]
	\centering
	\captionsetup[subfigure]{justification=centering}
	\begin{subfigure}[t]{0.3\columnwidth}
		\centering
		\includegraphics[width=\textwidth]{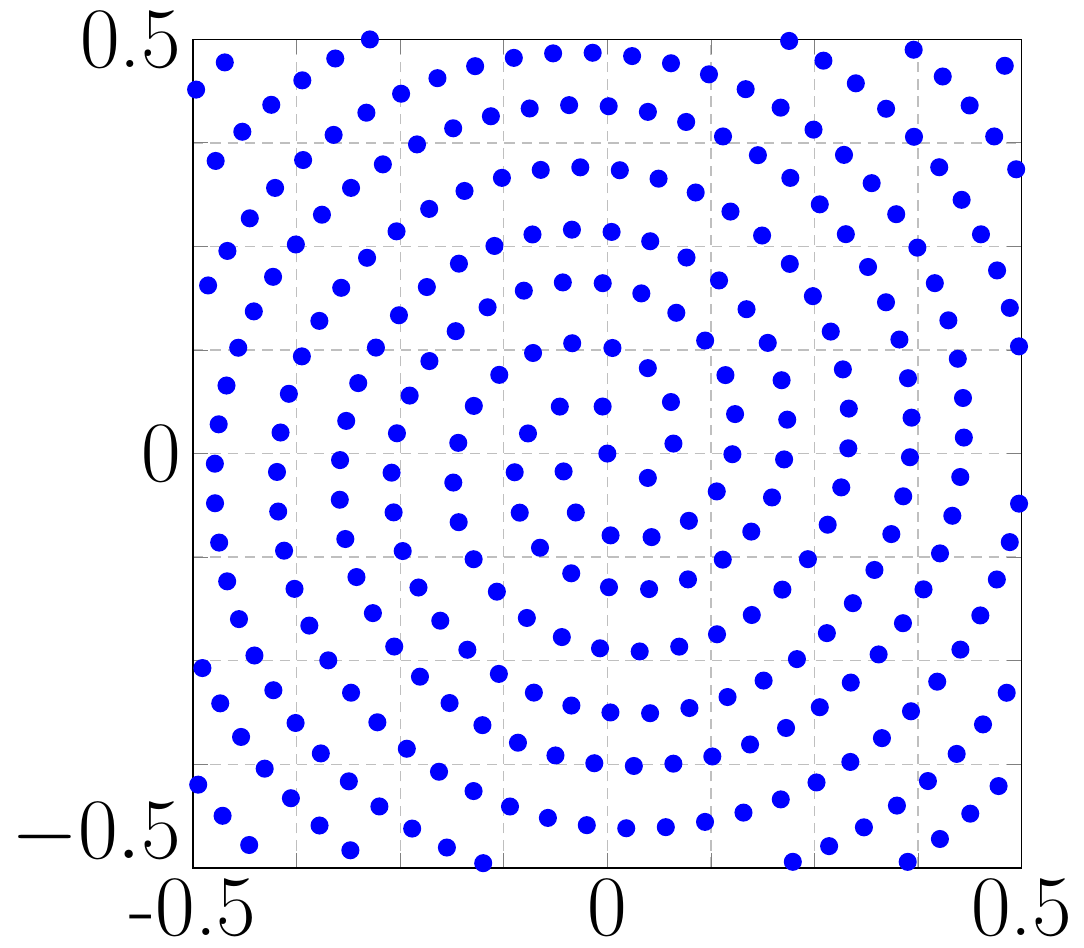}
		\caption{Spiral grid}
		\label{fig:polar_grids_spiral}
	\end{subfigure}
	\begin{subfigure}[t]{0.3\columnwidth}
		\centering
		\includegraphics[width=\textwidth]{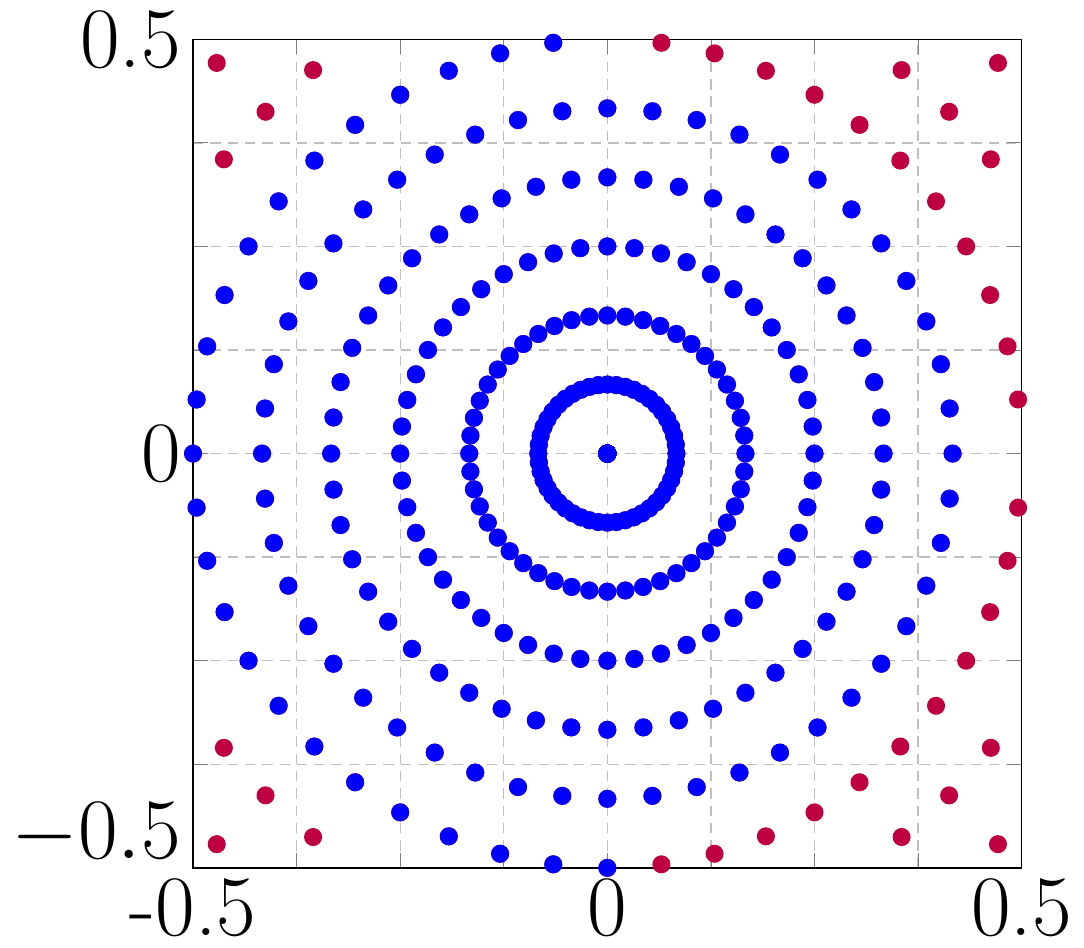}
		\caption{Polar (blue) and modified polar (red) grid}
		\label{fig:polar_grids_mpolar}
	\end{subfigure}
	\begin{subfigure}[t]{0.3\columnwidth}
		\centering
		\includegraphics[width=\textwidth]{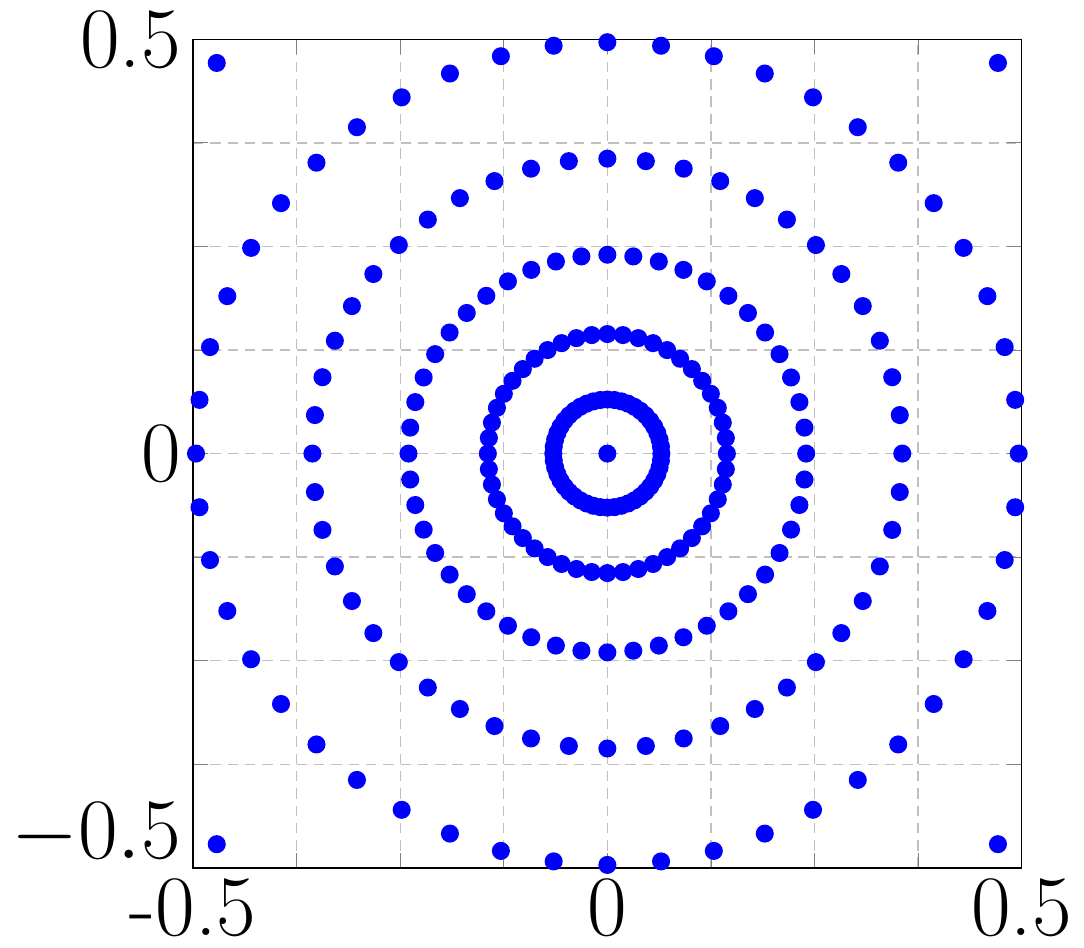}
		\caption{Logarithmic modified polar grid}
		\label{fig:polar_grids_log_mpolar}
	\end{subfigure}
	\caption{
		Exemplary grids of size $R=12$ and $T=2R$.
		\label{fig:polar_grids}}
\end{figure}

\begin{example}
	\label{ex:reconstr_phantom}
	Analogous to \cite[Ex.~5.4]{KiPo23} we firstly consider the reconstruction of trigonometric polynomials using the example of the Shepp-Logan phantom, see Fig.~\ref{fig:reconstr_phantom_spiral_original}, i.\,e., 
	the phantom data shall be treated as Fourier coefficients \mbox{$\b{\hat f}\coloneqq(\hat f_{\b k})_{\b k \in \I_{\b M}}$} of a trigonometric polynomial \eqref{eq:trig_poly_2d}.
	For given \mbox{$\b x_j \in \left[-\frac 12,\frac 12\right)^2$}, \mbox{$j=1,\dots,N$}, we then compute the evaluations $f(\b x_j)$ of \eqref{eq:trig_poly_2d} by means of an NFFT
	and use the resulting vector as input for the reconstruction.
	
	Since for \mbox{$|\I_{\b{2M}}| \leq N$} the optimality of weights computed by \eqref{eq:normal_equations_second_kind_double} was shown in \cite{KiPo23},
	we now consider \mbox{$|\I_{\b{2M}}|>N$} and compare the density compensation factors mentioned in Section~\ref{sec:trig_poly}.
	More precisely, for the spiral grid, cf.~Fig.~\ref{fig:polar_grids_spiral}, of size \mbox{$R=M$}, \mbox{$T=2R$}, we study 
	\eqref{eq:normal_equations_second_kind_double}, \eqref{eq:normal_equations_first_kind_double} and \eqref{eq:wcf_system_matrix}.
	The resulting reconstruction of the phantom of size \mbox{$M=64$} is presented in Fig.~\ref{fig:reconstr_phantom_spiral}~(top) including a detailed view of the 52nd row (bottom).	
	It can easily be seen that since the exactness condition \mbox{$|\I_{\b{2M}}| \leq N$} (see~Theorem~\ref{Thm:exact_trig_poly}) is violated, 
	the weights computed by \eqref{eq:normal_equations_second_kind_double} do not yield an exact reconstruction, cf.~Fig.~\ref{fig:reconstr_phantom_spiral_exact}.
	Note that the results using the least squares approximation by \eqref{eq:normal_equations_first_kind_double} are just as bad, cf.~Fig.~\ref{fig:reconstr_phantom_spiral_ls}.
	Merely, for the weights computed by \eqref{eq:wcf_system_matrix} we see a slight improvement in Fig.~\ref{fig:reconstr_phantom_spiral_wcf},	 
	since there are no artifacts on the inside of the phantom and the differences to the original phantom are not visible anymore.
	We remark that the comparatively small choice of \mbox{$M=64$} is necessary for \eqref{eq:wcf_system_matrix} being computationally affordable.
	\begin{figure}[ht]
		\centering
		\captionsetup[subfigure]{justification=centering}
		\begin{subfigure}[t]{0.24\columnwidth}
			\centering 
			\includegraphics[width=\textwidth,trim={0 0.85cm 0 0},clip]{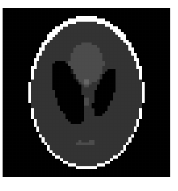}
		\end{subfigure}
		\begin{subfigure}[t]{0.24\columnwidth}
			\centering
			\includegraphics[width=\textwidth,trim={0 0.85cm 0 0},clip]{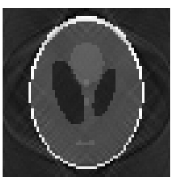}
		\end{subfigure}
		\begin{subfigure}[t]{0.24\columnwidth}
			\centering
			\includegraphics[width=\textwidth,trim={0 0.85cm 0 0},clip]{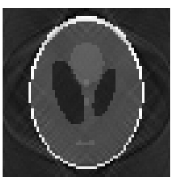}
		\end{subfigure}
		\begin{subfigure}[t]{0.24\columnwidth}
			\centering
			\includegraphics[width=\textwidth,trim={0 0.85cm 0 0},clip]{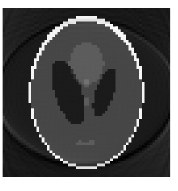}
		\end{subfigure}
		\begin{subfigure}[t]{0.24\columnwidth}
			\centering
			\includegraphics[width=\textwidth]{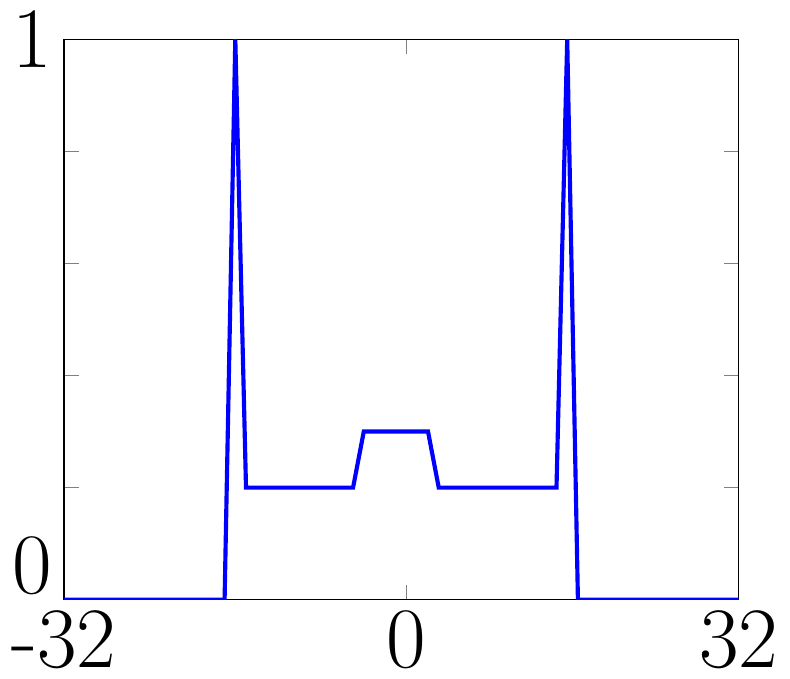}
			\caption{Phantom}
			\label{fig:reconstr_phantom_spiral_original}
		\end{subfigure}
		\begin{subfigure}[t]{0.24\columnwidth}
			\centering
			\includegraphics[width=\textwidth]{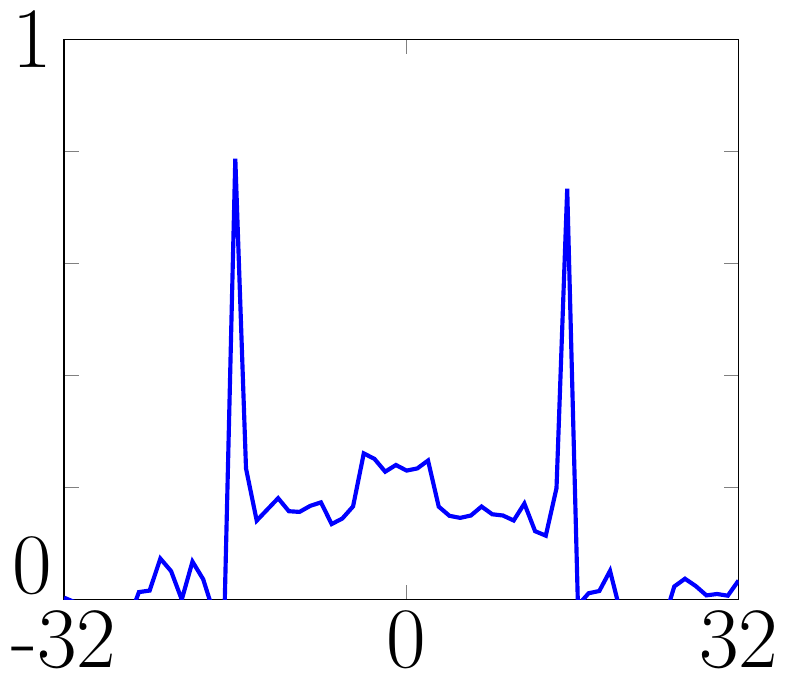}
			\caption{Use of \eqref{eq:normal_equations_second_kind_double}}
			\label{fig:reconstr_phantom_spiral_exact}
		\end{subfigure}
		\begin{subfigure}[t]{0.24\columnwidth}
			\centering
			\includegraphics[width=\textwidth]{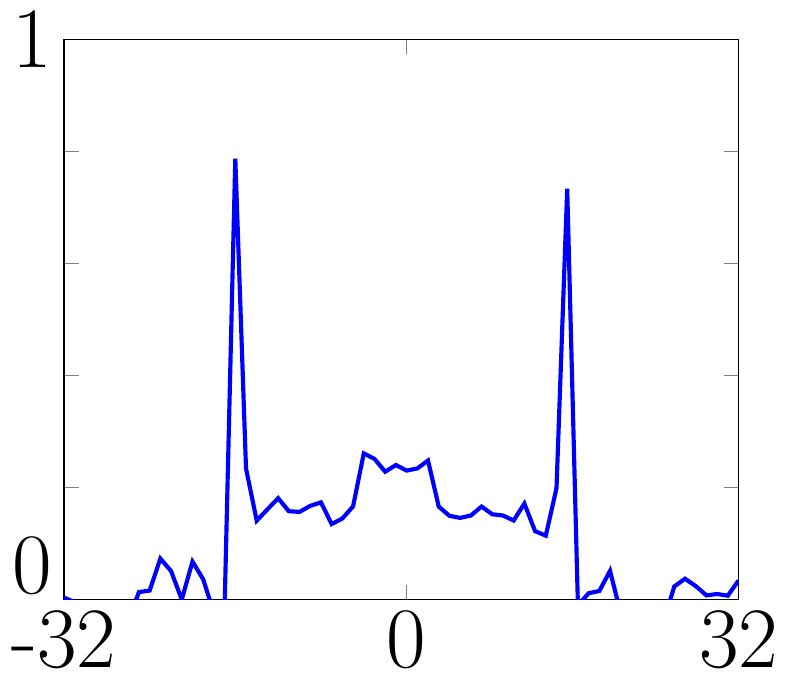}
			\caption{Use of \eqref{eq:normal_equations_first_kind_double}}
			\label{fig:reconstr_phantom_spiral_ls}
		\end{subfigure}
		\begin{subfigure}[t]{0.24\columnwidth}
			\centering
			\includegraphics[width=\textwidth]{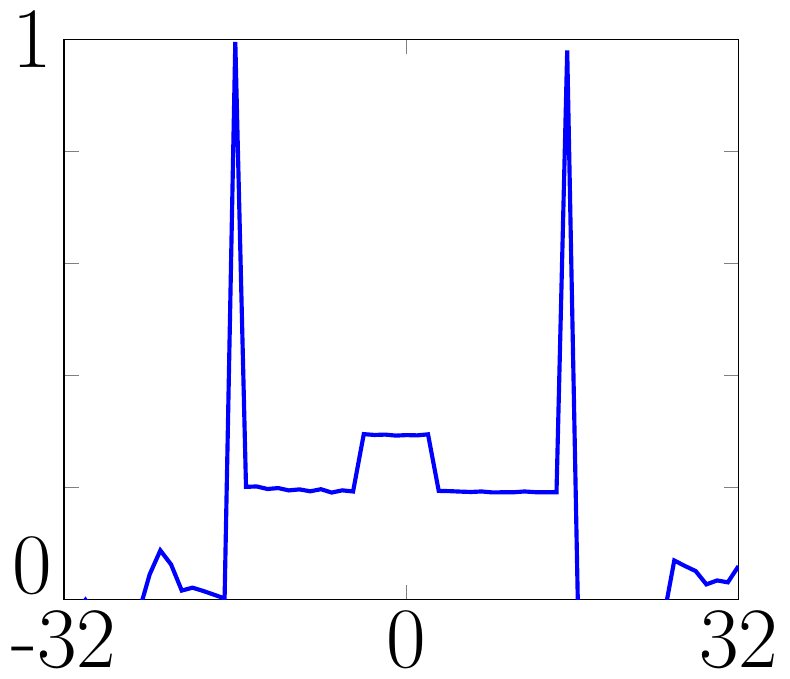}
			\caption{Use of\hspace{0.47ex}\eqref{eq:wcf_system_matrix}}
			\label{fig:reconstr_phantom_spiral_wcf}
		\end{subfigure}
		\caption{
			Reconstruction of the Shepp-Logan phantom of size \mbox{$M= 64$} (top) via density compensation factors computed by \eqref{eq:normal_equations_second_kind_double}, \eqref{eq:normal_equations_first_kind_double} and \eqref{eq:wcf_system_matrix}  for the spiral grid, cf.~Fig.~\ref{fig:polar_grids_spiral}, of size \mbox{$R=M$},\, \mbox{$T=2R$}; as well as a detailed view of the 52nd row (bottom).
			\label{fig:reconstr_phantom_spiral}}
	\end{figure}
	\ex
\end{example}

\begin{example}
	\label{ex:reconstr_bandlim}
	Summing up, we examine the reconstruction properties for bandlimited functions \mbox{$f\in L_1(\R^d)\cap C_0(\R^d)$} with maximum bandwidth~$\b M$.
	To determine the errors properly, we firstly specify a compactly supported function $\hat f$ and consequently compute its inverse Fourier transform \eqref{eq:forward_integral}, such that its samples $f(\b x_j)$ for given \mbox{$\b x_j \in \left[-\frac 12,\frac 12\right)^2$}, \mbox{$j=1,\dots,N$}, can be used for the reconstruction of the samples \mbox{$\hat f(\b k)$}, \mbox{$\b k\in\I_{\b{M}}$}.
	As in \cite[Ex.~5.5]{KiPo23} we consider the tensorized function \mbox{$\hat f(\b v) = g(v_1)\cdot g(v_2)$},
	where $g(v)$ is the one-dimensional triangular pulse 
	\mbox{$g(v) \coloneqq ( 1-\left|\tfrac{v}{b}\right| ) \cdot \chi_{[-b, b]}(v)$}.
	Then for all \mbox{$b\in\N$} with \mbox{$b\leq \frac M2$} 
	the associated inverse Fourier transform 
	\begin{align}
	\label{eq:test_func_bandlim}
	f(\b x)
	=
	\int_{\R^2} \hat f(\b v) \,\mathrm e^{2\pi\mathrm i \b v\b x} \,\mathrm{d}\b v 
	=
	b^2 \,\mathrm{sinc}^2(b\pi \b x) ,
	\ \;\b x\in\R^2 ,
	\end{align}
	is bandlimited with bandwidth $\b M$.
	For this test function, we compare the density compensation methods from Section~\ref{sec:bandlim}, i.\,e., 
	we study the computaion schemes
	\eqref{eq:normal_equations_second_kind_double}, \eqref{eq:wcf_system_matrix} and \eqref{eq:sol_greengard}.
	
	As a first experiment we fix \mbox{$M=32$} and \mbox{$b=12$}
	and consider the case \mbox{$|\I_{\b{2M}}| \leq N$}, which yields optimality for trigonometric polynomials.
	In addition to the real-world sampling of \eqref{eq:test_func_bandlim} also examine the artificial sampling data
	\begin{align}
	\label{eq:periodization_f_poly}
	\tilde f(\b x_j) 
	=
	\sum_{\b k\in\I_{\b{M}}} \hat f({\b k}) \,\e^{2\pi\i\b k\b x_j} 
	\end{align}
	of the periodization \eqref{eq:periodization_f}.
	A visualization of the chosen test function $f$ and the difference $f-\tilde f$ can be found in Fig.~\ref{fig:artificial_sampling_data}.
	\begin{figure}[!h]
		\centering
		\captionsetup[subfigure]{justification=centering}
		\begin{subfigure}[t]{0.45\columnwidth}
			\centering
			\includegraphics[width=\textwidth]{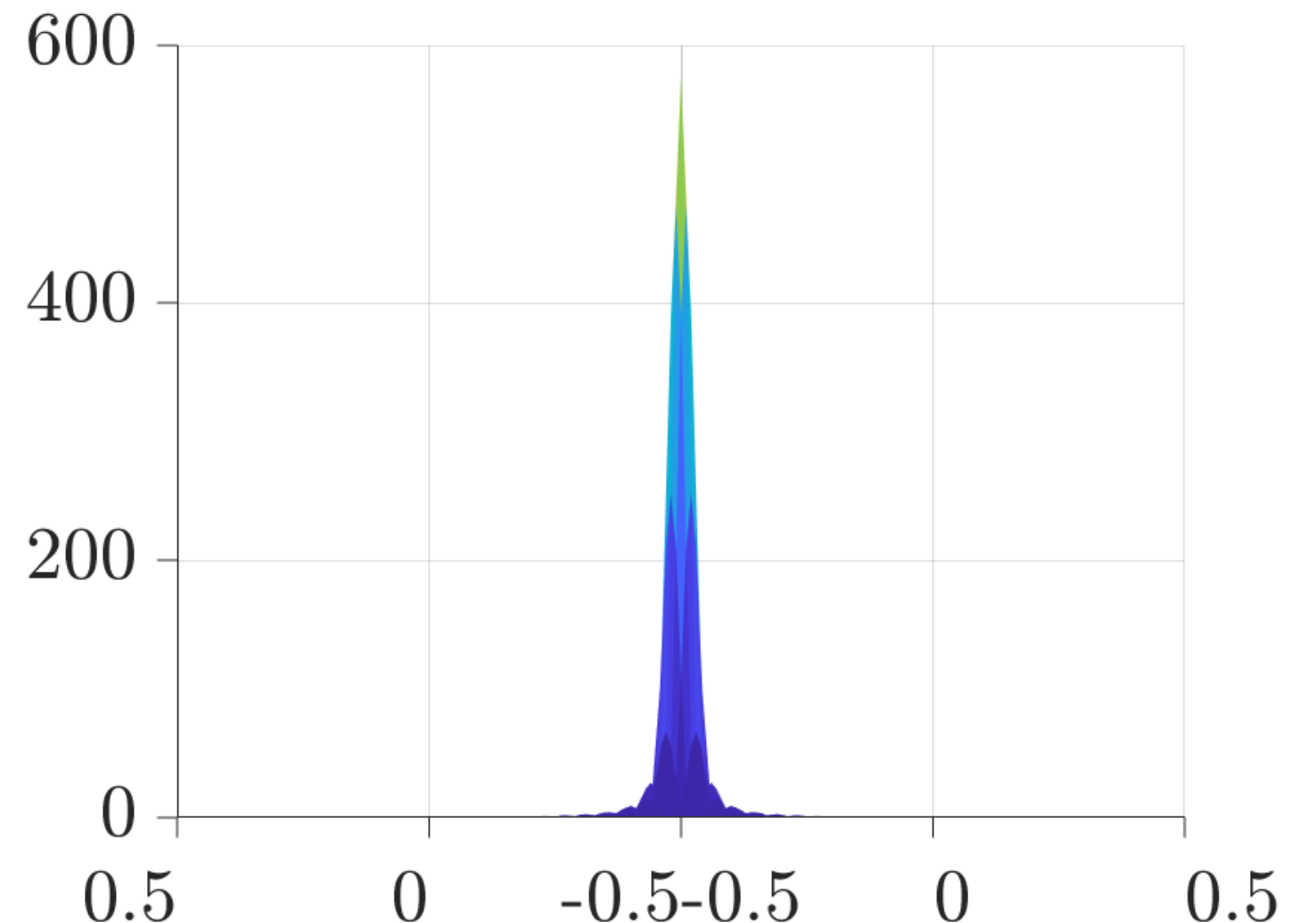}
			\caption{Test function $f$}
			\label{fig:artificial_sampling_data_original}
		\end{subfigure}
		\begin{subfigure}[t]{0.45\columnwidth}
			\centering
			\includegraphics[width=\textwidth]{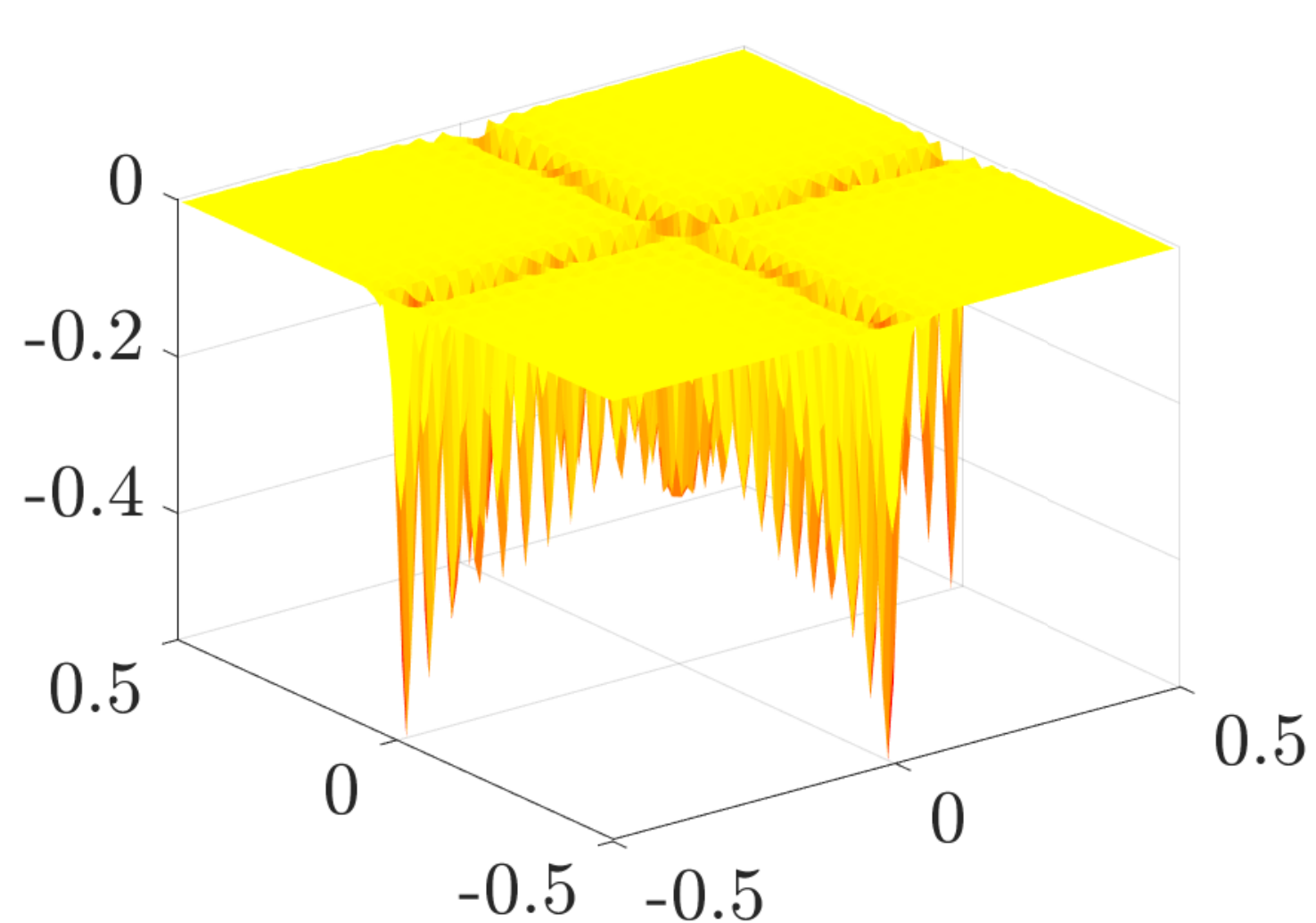}
			\caption{$f-\tilde f$}
			\label{fig:artificial_sampling_data_error}
		\end{subfigure}
		\caption{
			The test function \eqref{eq:test_func_bandlim} and its periodization \eqref{eq:periodization_f_poly}.
			\label{fig:artificial_sampling_data}}
	\end{figure}

	For the modified polar grid, cf.~Fig.~\ref{fig:polar_grids_mpolar}, of size \mbox{$R=2M$}, \mbox{$T=2R$}, we use these two kinds of sampling data to compute the
	reconstructions \mbox{$\b{\tilde h} \coloneqq ({\tilde h}^{\mathrm{w}}_{\b k})_{\b k\in\I_{\b{M}}}$}, cf.~\eqref{eq:Fourier_coeffs_nonequi}, and the pointwise  errors \mbox{$\big|\b{\tilde h} - \b{\hat f}\big|$}.
	The corresponding results are displayed in Fig.~\ref{fig:reconstr_bandlim_mpolar}.
	It can easily be seen that for the artificial sampling data \eqref{eq:periodization_f_poly} the weights computed by~\eqref{eq:normal_equations_second_kind_double} indeed yield an exact reconstruction, see Fig.~\ref{fig:reconstr_bandlim_mpolar_exact_tilde} (bottom), and thus are optimal.
	However, in the more realistic setting the results are not as good, but the weights by \eqref{eq:normal_equations_second_kind_double} and~\eqref{eq:wcf_system_matrix}, see Fig.~\ref{fig:reconstr_bandlim_mpolar_exact_tilde} and \ref{fig:reconstr_bandlim_mpolar_wcf_tilde} (top),
	produce nearly the same error as a reconstruction on an equispaced grid.
	Hence, for bandlimited functions the truncation \mbox{$\tilde f(\b x_j) \approx f(\b x_j)$} is the dominating error term and therefore reconstruction errors smaller than the ones shown in Fig.~\ref{fig:reconstr_bandlim_mpolar} cannot be expected.
	In addition, we remark that \eqref{eq:sol_greengard} is not as good, for none of the sampling data, see Fig.~\ref{fig:reconstr_bandlim_mpolar_green_tilde}.
	Note that we chose comparatively small \mbox{$M=32$} in order that the computation schemes \eqref{eq:wcf_system_matrix} and \eqref{eq:sol_greengard} are affordable.
	\begin{figure}[ht]
		\centering
		\captionsetup[subfigure]{justification=centering}
		\begin{subfigure}[t]{0.32\columnwidth}
			\centering 
			\includegraphics[width=\textwidth,trim={4cm 8cm 2.7cm 0},clip]{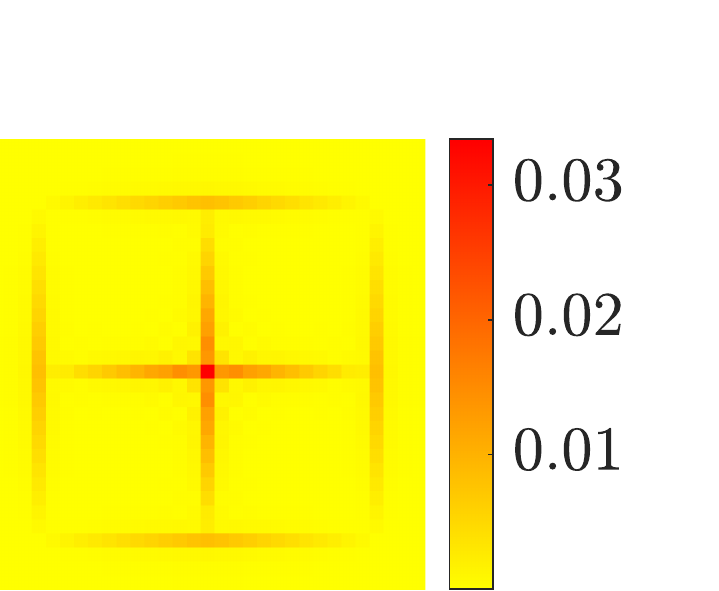}
		\end{subfigure}
		\begin{subfigure}[t]{0.32\columnwidth}
			\centering 
			\includegraphics[width=\textwidth,trim={4cm 8cm 2.7cm 0},clip]{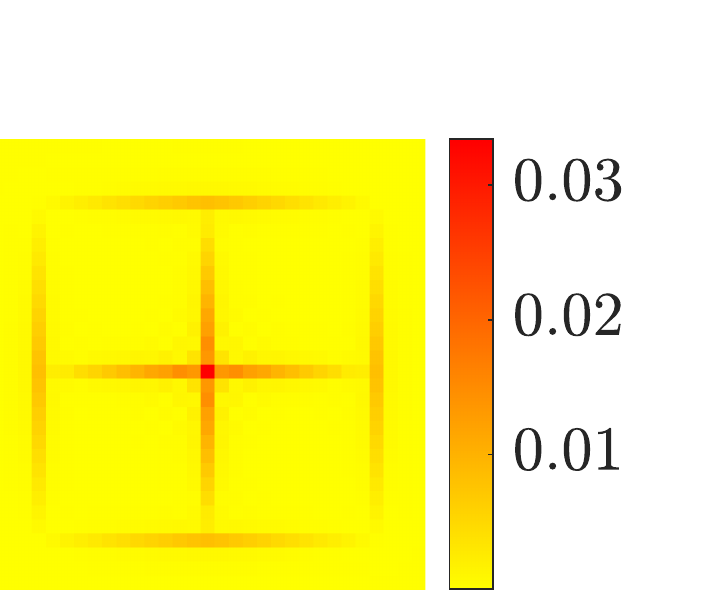}
		\end{subfigure}
		\begin{subfigure}[t]{0.32\columnwidth}
			\centering 
			\includegraphics[width=\textwidth,trim={4cm 8cm 2.7cm 0},clip]{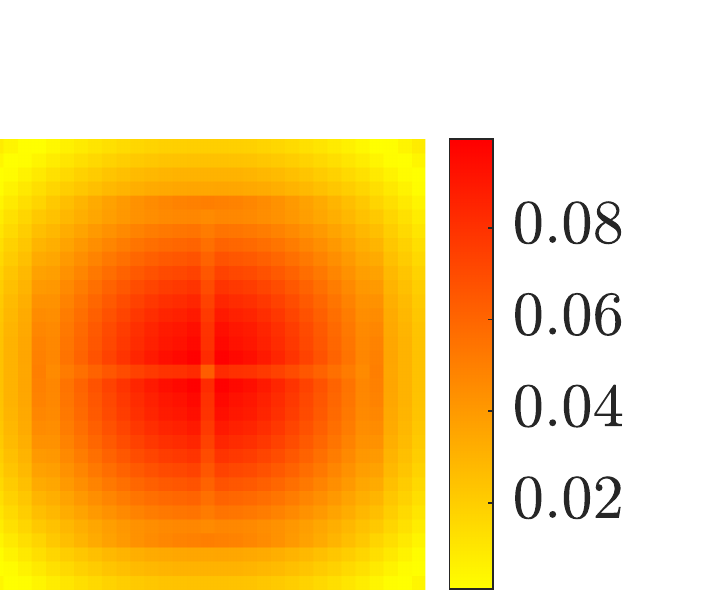}
		\end{subfigure}
		\begin{subfigure}[t]{0.32\columnwidth}
			\centering
			\includegraphics[width=\textwidth,trim={4cm 8cm 2.7cm 0},clip]{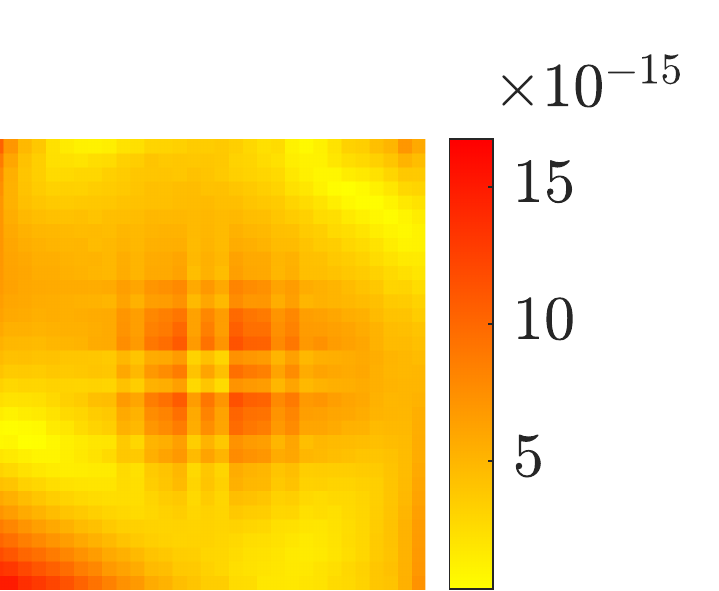}
			\caption{Use of \eqref{eq:normal_equations_second_kind_double}}
			\label{fig:reconstr_bandlim_mpolar_exact_tilde}
		\end{subfigure}
		\begin{subfigure}[t]{0.32\columnwidth}
			\centering
			\includegraphics[width=\textwidth,trim={4cm 8cm 2.7cm 0},clip]{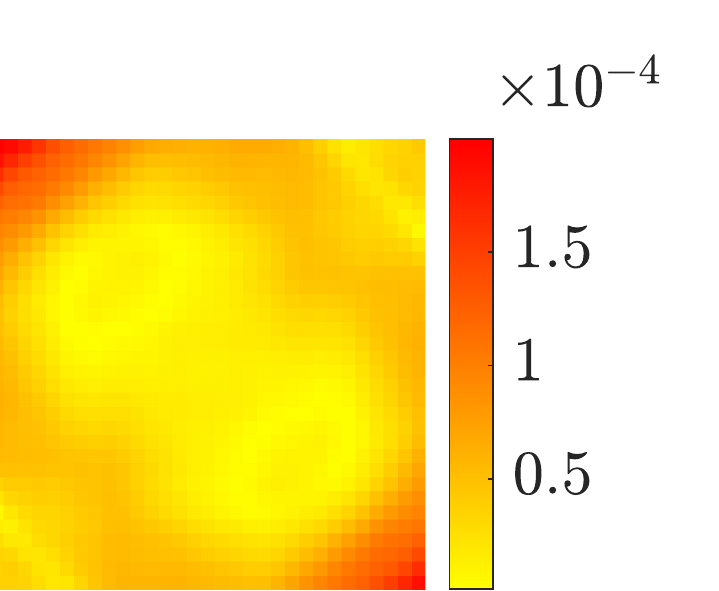}
			\caption{Use of \eqref{eq:wcf_system_matrix}}
			\label{fig:reconstr_bandlim_mpolar_wcf_tilde}
		\end{subfigure}
		\begin{subfigure}[t]{0.32\columnwidth}
			\centering
			\includegraphics[width=\textwidth,trim={4cm 8cm 2.7cm 0},clip]{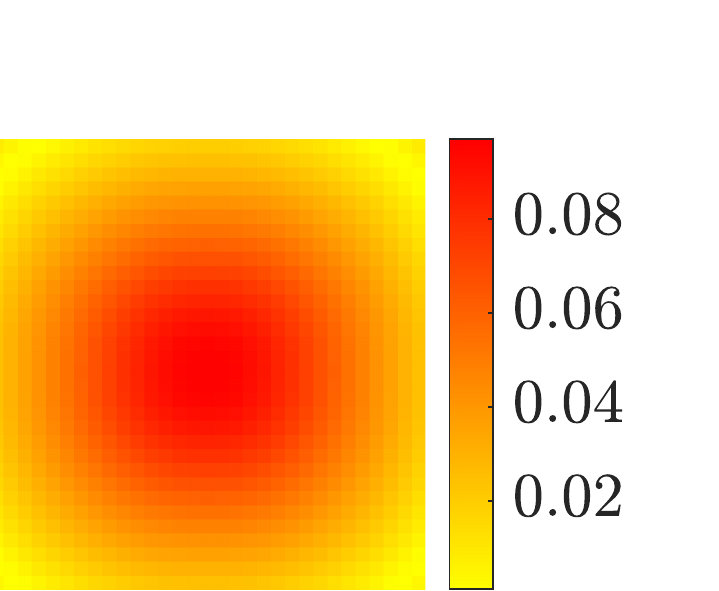}
			\caption{Use of \eqref{eq:sol_greengard}}
			\label{fig:reconstr_bandlim_mpolar_green_tilde}
		\end{subfigure}
		\caption{
			Pointwise error of the reconstruction of the tensorized triangular pulse with \mbox{$M=32$} and \mbox{$b=12$}, via density compensation factors computed by \eqref{eq:normal_equations_second_kind_double}, \eqref{eq:wcf_system_matrix} and \eqref{eq:sol_greengard} for the modified polar grid, cf.~Fig.~\ref{fig:polar_grids_mpolar}, of size \mbox{$R=2M$}, \mbox{$T=2R$}, using samples $f(\b x_j)$ (top) and artificial samples $\tilde f(\b x_j)$ (bottom).
			\label{fig:reconstr_bandlim_mpolar}}
	\end{figure}

	In a second experiment we sample the function \eqref{eq:test_func_bandlim} with \mbox{$M=64$} and \mbox{$b=24$} at
	logarithmic modified polar grids, cf.~Fig.~\ref{fig:polar_grids_log_mpolar}, of different sizes \mbox{$R$}, \mbox{$T=2R$}.
	The corresponding relative errors 
	\begin{align}
	\label{eq:relative_errors}
	\frac{\|\b{\tilde h}-\b{\hat f}\|_{2}}{\|\b{\hat f}\|_{2}}
	\end{align}
	can be found in Table~\ref{table:reconstr_bandlim_logmpolar}.
	Note that for \mbox{$M=64$} we have \mbox{$|\I_{\b M}|=4096$} and \mbox{$|\I_{\b {2M}}|=16384$}.
	We observe that for \mbox{$|\I_{\b {2M}}| \leq N$} the errors from \eqref{eq:normal_equations_second_kind_double} and \eqref{eq:wcf_system_matrix} are comparably good, while except for \mbox{$|\I_{\b {M}}| > N$} the scheme \eqref{eq:wcf_system_matrix} produces optimal results.
	\begin{table}[ht]
		\centering
		\begin{tabular}{|c|c|c|c|c|}
			\hline
			R & N & Use of \eqref{eq:normal_equations_second_kind_double}  & Use of \eqref{eq:wcf_system_matrix} & Use of \eqref{eq:sol_greengard} \\
			\hline\hline
			40 & 3565 & 4.4908e-01 & 1.7608e-01 & 2.0475e-01 \\ 
			\hline
			48 & 5145 & 1.0886e-01 & 2.0690e-02 & 1.5829e-01 \\ 
			56 & 7149 & 3.6632e-02 & 8.0215e-03 & 1.5401e-01 \\ 
			\hline
			64 & 9429 & 2.5109e-02 & 4.7988e-03 & 1.8337e-01 \\ 
			72 & 11965 & 7.6871e-03 & 4.1096e-03 & 2.0633e-01 \\ 
			80 & 14909 & 5.5991e-03 & 3.8507e-03 & 2.1932e-01 \\ 
			\hline
			88 & 18153 & 3.8889e-03 & 3.9853e-03 & 2.2665e-01 \\ 
			96 & 21589 & 4.2240e-03 & 3.7917e-03 & 2.3092e-01 \\ 
			\hline
		\end{tabular}
		\caption{
			Relative errors~\eqref{eq:relative_errors} of the reconstruction of the tensorized triangular pulse with \mbox{$M=64$} and \mbox{$b=24$}, via density compensation factors computed by \eqref{eq:normal_equations_second_kind_double}, \eqref{eq:wcf_system_matrix} and \eqref{eq:sol_greengard} for log. modified polar grids, cf.~Fig.~\ref{fig:polar_grids_log_mpolar}, of different sizes.
			\label{table:reconstr_bandlim_logmpolar}}
	\end{table}
	\ex
\end{example}

\section*{Acknowledgments}
Melanie Kircheis acknowledges the support from the BMBF grant 01$\mid$S20053A (project SA$\ell$E).
Daniel Potts acknowledges the funding by Deutsche Forschungsgemeinschaft (German Research Foundation) -- Project--ID 416228727 -- SFB 1410.

\bibliographystyle{abbrv}

\end{document}